\documentclass[12pt]{amsart}
\usepackage{graphicx}
\usepackage{geometry}
\usepackage{amssymb, amsmath, amsfonts}
\usepackage{url}
\usepackage[pdfborder={0 0 0}]{hyperref}
\usepackage{verbatim} 
\usepackage[all]{xy}
\usepackage{color}

\usepackage{marginnote}

\usepackage{breakurl}

\allowdisplaybreaks
\newtheorem{theorem}{Theorem}[section]

\newtheorem{lemma}[theorem]{Lemma}
\newtheorem{definition}[theorem]{Definition}

\newtheorem{remark}[theorem]{Remark}

\numberwithin{equation}{section}

\newif\ifcolor
\colortrue  %  <-----------ACTIVATE  IF COLOR
\keywords{Dominative $p$-Laplacian, viscosity solutions, Power concavity.} \subjclass[2010]{35J15, 35J60.}

\begin{document}
	
	\raggedbottom
	
	\title{Concave power solutions of the Dominative $p$-Laplace equation}
	
	\author[F. H\o eg]{Fredrik Arbo H\o eg}
	\address{Department of Mathematical Sciences\\
		Norwegian University of Science and Technology\\
		NO-7491 Trondheim\\ Norway}
	\email{fredrik.hoeg@ntnu.no}

	\maketitle
	
	% \subjclass[2000]{Primary: 35L70; Secondary: 49K20}
	
	\begin{abstract}
		In this paper, we study properties of solutions of the Dominative $p$-Laplace equation with homogeneous Dirichlet boundary conditions in a bounded convex domain $\Omega$. For the equation $-\mathcal{D}_p u= 1$, we show that $\sqrt{u}$ is concave, and for the eigenvalue problem $\mathcal{D}_p u + \lambda u=0$, we show that $\log {u}$ is concave. 
	\end{abstract}

	\section{Introduction}	
The \textit{Dominative} $p$-Laplace operator was defined as 
\begin{align*}
\mathcal{D}_p u := \Delta u  + (p-2)\lambda_{\max} (D^2 u), \quad p \geq 2, 
\end{align*}	
by Brustad in [B1] and later studied in [B2]. See also [BLM] for a stochastic interpretation and a game-theoretic approach of the equation. Here, $\lambda_{\max}$ denotes the largest eigenvalue of the Hessian matrix
\begin{align*}
D^2 u = \left( \frac{\partial^2 u}{\partial x_i \partial x_j}   \right)_{ij}.
\end{align*}
We shall study the two equations 
\begin{align*}
-\mathcal{D}_p u = 1 \quad \text{and} \quad \mathcal{D}_p u + \lambda u =0
\end{align*}
in a bounded convex domain $\Omega \subset \mathbb{R}^n$. The positive solutions with zero boundary values have the property for $-\mathcal{D}_p u =1$ that $\sqrt{u}$ is concave, see Theorem \ref{maintheorem} below. In Theorem \ref{maintheoremeigenvalue} we show that for $\mathcal{D}_p u + \lambda u =0$,  $\log u $ is concave. Problems related to concave solutions have been studied for $p$-Laplace type equations, and we give a quick review of the results. The operator is closely related to the \textit{normalized} $p$-Laplace operator,
\begin{align*}
\Delta_p^N u = |\nabla u |^{2-p} \text{div} \left (  |\nabla u|^{p-2} \nabla u   \right ),
\end{align*}
which describes a Tug-of-war game with noise, see [MPR]. Due to this, the operator has been studied extensively over the last 15 years, and we refer to [D],[HL], [APR] for an introduction and some regularity results. The solutions are weak and appear in the form of \textit{viscosity solutions} and we refer to [CIL] for an introduction of viscosity solutions.
If $u$ is a solution of the problem
	\begin{align*}
	\begin{cases}
	-\Delta u = 1 \quad &\text{in} \, \, \Omega\\
	u=0 \quad &\text{on} \, \, \partial \Omega
	\end{cases}
	\end{align*}
in a bounded convex domain $\Omega \subset \mathbb{R}^n$, one can show that $\sqrt{u}$ is concave. This problem, including more complex right-hand sides, was studied in the 1970's and 1980's by [Ka],[Ke],[Ko] and [M]. For $n=1$ and $n=2$, a brute force calculation shows that $\sqrt{u}$ is concave. For $n \geq 3$ the proofs are more complicated. For the ordinary $p$-Laplacian, [S] showed that  $u^{\frac{p-1}{p}}$ is concave. One should note that simply setting $p=2$ does not simplify the proof. Thus, the papers [Ke] and [Ko] are still of great value. For the infinity Laplacian, $\Delta_\infty u= \left \langle D^2 u \nabla u, \nabla u  \right \rangle$, [CF] showed that $u^\frac{3}{4}$ is concave. Our result for the Dominative $p$-Laplace equation can be formulated in the following theorem. We say that $\Omega$ satisfies the interior sphere condition if for all $y \in \partial \Omega$ there is an $x \in \Omega$ and an open ball $B_r(x)$ such that $B_r(x) \subset \Omega$ and $y \in \partial B_r(x)$. 

\begin{theorem}
	Let $u \in C(\bar{\Omega})$ be a viscosity solution of 
	\begin{align*}
	\begin{cases}
	-\mathcal{D}_p u= 1 \quad  &\text{in} \, \, \Omega \\
	\quad u= 0 \quad &\text{on} \, \,  \partial \Omega
	\end{cases}
	\end{align*}
	in a bounded convex domain $\Omega \subset \mathbb{R}^n$ which satisfies the interior sphere condition. Then $\sqrt{u}$ is concave. 
	\label{maintheorem}
\end{theorem}
 
Further, we study the eigenvalue problem and give the following result.

\begin{theorem}
	Let $u\in C(\bar{\Omega})$ be a positive viscosity solution of 
	\begin{align*}
	\begin{cases}
	-\mathcal{D}_p u = \lambda u \quad  &\text{in} \, \, \Omega \\
	u=0 \quad &\text{on} \, \, \partial \Omega 
	\end{cases}
	\end{align*}
	with $\lambda >0$ in a bounded convex domain $\Omega \subset{R}^n$. Then $\log u $ is concave. 
	\label{maintheoremeigenvalue}
\end{theorem}

\begin{remark}
We give a remark on what happens when $p \rightarrow \infty$ in Theorem \ref{maintheorem}. After dividing the equation by $p$ and letting $p$ approach infinity, the following equation is obtained
\begin{align*}
\begin{cases}
-\lambda_{\max}(D^2 u )= 0 \quad  &\text{in} \, \, \Omega \\
\quad u= 0 \quad &\text{on} \, \,  \partial \Omega.
\end{cases}
\end{align*}
This equation has the solution $u=0$, which is obviously already concave. This is better than the square root being concave, so for $p=\infty$ a stronger result is obtained. (For a less trivial result, another normalization with $p$ is needed.)

\end{remark}
For the Helmholtz equation $\Delta u + \lambda u=0$, the problem related to concave logarithmic solutions has been studied in [BL], [Ko] and [CS].
The nonlinear eigenvalue problem associated with the $p$-Laplace equation has been studied for example in [L] and [S]. In [S], Sakaguchi showed that $\log{u}$ is a concave function.

%For $n=1$, the above problem in $[-1,1]$ have solutions, for $n \in \mathbb{Z}$,  $u_n=A_n \cos{(n+\frac{1}{2})\pi x}$ when $\lambda_n= (p-1)(n+\frac{1}{2})^2 \pi^2 $ and one can see directly that the square root is a concave function. For radial functions, $u(x)=u(|x|)$, [KKK] showed that the eigenvalue problem for the normalized $p$-Laplace equation has solutions which can be written as Bessel functions, and they form a complete orthonormal system in a weighted $L^2$-space. Here, the result is the same for the Dominative $p$-Laplace equation, since $\Delta_p^N u =\mathcal{D}_p u$ when $u$ is radial. 

\subsection*{Acknowledgements.} I wish to thank the referee for valuable comments.

\section{Preliminaries and notation}

The gradient of a function $f: \Omega_T \rightarrow \mathbb{R}$ is \\ $$\nabla f= \left( \frac{\partial f}{\partial x_1},...,\frac{\partial f}{\partial x_n} \right) $$ %\\[1em]
and its Hessian matrix is  \\ $$\left( D^2 f\right)_{ij} =\frac{\partial^2 f}{\partial x_i \partial x_j}.$$
We will use the operator 
\begin{align*}
\mathcal{D}_p u = \Delta u + (p-2)\lambda_{\max} (D^2 u)\\
\end{align*}
and if applied to a matrix $X \in S^n$, we use
\begin{align*}
D_p X = \text{tr}(X) + (p-2)\lambda_{\max} (X). 
\end{align*}
Also, the \textit{normalized} $p$-Laplace operator is referred to,
\begin{align*}
\Delta_p^N u = \Delta u + (p-2)\cdot\frac{1}{|\nabla u|^2 } \sum_{i,j=1}^N u_{x_i} u_{x_j} u_{x_i x_j}.
\end{align*}

\textbf{Viscosity solutions.} The Dominative $p$-Laplace operator is uniformly elliptic. Therefore, it is convenient to use viscosity solutions as a notion of weak solutions. Throughout the text, we always keep $p\geq2$. In the definition below, $g$ is assumed to be continuous in all variables. 

\begin{definition}
A function $u \in USC(\bar{\Omega})$ is a viscosity subsolution to ${-\mathcal{D}_p u=g(x,u, \nabla u)}$ if, for all $\phi \in C^2(\Omega)$,
\begin{align*}
-\mathcal{D}_p  \phi (x) \leq g(x,u,\nabla \phi).
\end{align*}
at any point $x \in \Omega$ where $u-\phi$ attains a local maximum. 
A function $u \in LSC(\bar{\Omega})$ is a viscosity supersolution to $-\mathcal{D}_p u=g(x,u,\nabla u)$ if, for all $\phi \in C^2(\Omega)$, 
\begin{align*}
-\mathcal{D}_p \phi (x) \geq g(x,u, \nabla \phi). 
\end{align*} 
at any point $x \in \Omega$ where $u-\phi$ attains a local minimum.\\
\\
\\
A function $u\in C(\bar{\Omega})$ is a viscosity solution of 
\begin{align*}
\begin{cases}
-\mathcal{D}_p u= g(x,u,\nabla u) \quad &\text{in} \, \, \Omega \\ 
\quad u= 0 \quad &\text{on} \, \,  \partial \Omega
\end{cases}
\end{align*}
if it is a viscosity sub- and supersolution of $-\mathcal{D}_p u= g(x,u,\nabla u)$ and $u= 0$ on $\partial \Omega$. 
\end{definition}
When defining viscosity solutions to $\-\Delta_p^N u = g(x,u,\nabla u)$, one has to be careful at points where the gradient vanishes. 
\begin{definition}
	
	A function $u \in USC(\bar{\Omega})$ is a viscosity subsolution of $-\Delta^N_p u=1$ if, for all $\phi \in C^2(\Omega)$,
	\begin{align*}
	\begin{cases}
	-\Delta^N_p   \phi (x) \leq 1 , \quad &\text{if} \,  \, \nabla \phi (x) \neq 0\\
	-\mathcal{D}_p \phi (x) \leq 1, \quad &\text{if} \, \, \nabla \phi (x)=0.
	\end{cases}  
	\end{align*}
	at any point $x \in \Omega$ where $u-\phi$ attains a local minimum. 
	A function $u \in LSC(\bar{\Omega})$ is a viscosity supersolution of $-\Delta^N_p  u=1$ if, for all $\phi \in C^2(\Omega)$, 
	\begin{align*}
	\begin{cases}
	-\Delta^N_p \phi (x) \geq 1, \quad &\text{if} \, \, \nabla \phi (x) \neq 0 \\
	\mathcal{D}_p (-\phi(x)) \geq 1, \quad &\text{if} \, \, \nabla \phi (x)=0.  
	\end{cases}
	\end{align*} 
	at any point $x \in \Omega$ where $u-\phi$ attains a local minimum. 
	A function $u\in C(\bar{\Omega})$ is a viscosity solution of 
	\begin{align*}
	\begin{cases}
	-\Delta^N_p u= 1 \quad &\text{in} \, \, \Omega \\ 
	\quad u= 0 \quad &\text{on} \, \,  \partial \Omega
	\end{cases}
	\end{align*}
	if it is a viscosity sub- and supersolution of $-\Delta^N_p u= 1$ and $u= 0$ on $\partial \Omega$. 
\end{definition}

We also need an equivalent definition of viscosity solutions using the \textit{sub-} and \textit{superjets}. For functions $u: \Omega \rightarrow \mathbb{R}^n$ they are given by

\begin{align*}
J^{2,+}u(x) = \Big \{ (q,X) \in \mathbb{R}^n \times S^n  : u(y) & \leq u(x) + \left \langle q, y-x \right \rangle \\
&+\frac{1}{2}\left \langle X(y-x), y-x \right \rangle  + o(|y-x|^2|) \, \text{as} \, \, y \rightarrow x  \Big \}
\end{align*}
and
\begin{align*}
J^{2,-}u(x) = \Big \{ (q,X) \in \mathbb{R}^n \times S^n  : u(y) & \geq u(x) + \left \langle q, y-x \right \rangle \\
&+\frac{1}{2}\left \langle X(y-x), y-x \right \rangle  + o(|y-x|^2|) \, \text{as} \, \, y \rightarrow x  \Big \}.
\end{align*}

\begin{definition}
A function $u \in USC(\bar \Omega)$ is a viscosity subsolution to ${-\mathcal{D}_p u =g(x,u, \nabla u)}$ if $(q,X) \in J^{2,+}u(x)$ implies
\begin{align*}
-D_p X \leq g(x,u,q).
\end{align*} 
A function $u \in USC(\bar \Omega)$ is a viscosity subsolution of $-\mathcal{D}_p u =g(x,u, \nabla u)$ if ${(q,X) \in J^{2,-}u(x)}$ implies
\begin{align*}
-D_p X \geq g(x,u,q).
\end{align*} 
A function $u\in C(\bar{\Omega})$ is a viscosity solution of 
\begin{align*}
\begin{cases}
-\mathcal{D}_p u= g(x,u,\nabla u) \quad &\text{in} \, \, \Omega \\ 
\quad u= 0 \quad &\text{on} \, \,  \partial \Omega
\end{cases}
\end{align*}
if it is a viscosity sub- and supersolution of $-\mathcal{D}_p u= g(x,u,\nabla u)$ and $u= 0$ on $\partial \Omega$. 

\end{definition}

We mention some results in [K] obtained for the normalized $p$-Laplace equation, which we will use together with the relationship between the normalized $p$-Laplace equation and the Dominative $p$-Laplace equation.

\begin{lemma}
	A function $u\in USC(\bar{\Omega})$ is a positive viscosity subsolution of $-\Delta^N_p u =1$ with $u= 0$ on $\partial \Omega$ if and only if $v=-\sqrt{u} \in LSC(\bar \Omega)$ is a negative viscosity supersolution of 
	\begin{align*}
	-\Delta^N_p v= \frac{1}{v}\left( (p-1)|\nabla v|^2 + \frac{1}{2}  \right). 
	\end{align*}
	\label{squaresolution}
\end{lemma}

\begin{lemma}
	Let $\lambda>0$. A function $u\in USC(\bar{\Omega})$ is a positive viscosity subsolution of ${-\Delta^N_p u =\lambda u}$ if and only if $v=-\ln {u} \in LSC(\bar{\Omega})$ is a negative viscosity supersolution of 
	\begin{align*}
	-\Delta^N_p v=-(p-1)|\nabla v|^2 - \lambda. 
	\end{align*}
	\label{logsolns}

\end{lemma}

\textbf{Properties of the operator.}
We give some properties of viscosity solutions of the Dominative $p$-Laplace equation. 

\begin{itemize}
	\item{\textbf{Comparison principle:} Let $u\in USC(\bar{\Omega})$ be a viscosity subsolution of $-\mathcal{D}_p u =1$ and let $v\in LSC(\bar{\Omega})$ be a viscosity supersolution of $-\mathcal{D}_p v= 1$. Then $u \leq v$ on $\partial \Omega$ implies $u \leq v$ in $\Omega$. For a proof, see [Theorem 3.3, CIL]. }
	\item{\textbf{Positive supersolutions:} If $u \in LSC(\bar{\Omega})$ is a viscosity supersolution of $-\mathcal{D}_p u=1$ with $u=0$ on $\partial \Omega$, then $u > 0 $ in $\Omega$. To see this, note that $w=0$ is a viscosity subsolution, and $u \geq w$ by the comparison principle. This inequality must be strict. If $u(x_0)=0$, then $x_0$ is a minimum for $u$. Let $\phi(x)=u(x_0)$ be a test function. Then $u-\phi$ has a local minimum at $x_0$. But $-\mathcal{D}_p \phi =0 <1$, which contradicts $u$ being a supersolution. }
\end{itemize}

The Dominative $p$-Laplace operator has many of the same properties that the normalized $p$-Laplace operator possess. Here, we give some connections for viscosity solutions. 

\begin{lemma}
	If $u\in LSC(\bar{\Omega})$ is a viscosity supersolution of 
	\begin{align*}
	-\mathcal{D}_p u = g(x, u, \nabla u),
	\end{align*}
	then $u$ is a viscosity supersolution of 
	\begin{align*}
	-\Delta^N_p u = g(x,u, \nabla u).
	\end{align*}
	Here, $g$ is assumed to be continuous in all variables. Similarly, if $u\in USC(\bar{\Omega})$ is a viscosity subsolution of $-\Delta^N_p u = g(x,u,\nabla u)$, then $u$ is a viscosity supersolution of $-\mathcal{D}_p u = g(x,u,\nabla u)$.
	\label{connectionvisc}
\end{lemma}
\begin{proof}
	Assume $u$ is a viscosity supersolution of $-\mathcal{D}_p u = g(x,u,\nabla u)$. If $u-\phi$ obtains a minimum at $x\in \Omega$, we have, provided $\nabla \phi(x) \neq 0$, 
	\begin{align*}
	-\Delta^N_p \phi \geq -\mathcal{D}_p \phi  \geq g(x, u, \nabla \phi).
	\end{align*}  
	If $\nabla \phi (x)=0$,
	\begin{align*}
	-\Delta \phi  - (p-2)\lambda_{\min} (D^2 \phi) \geq -\mathcal{D}_p \phi \geq g(x, u, \nabla \phi).
	\end{align*}
	Hence, $u$ is a viscosity supersolution of $-\Delta^N_p u = g(x,u, \nabla u)$. 
	If $u$ is a viscosity subsolution of $-\Delta^N_p u = g(x,u, \nabla u)$ and $u-\phi$ obtains a maximum at $x \in \Omega$, 
	\begin{align*}
	-\mathcal{D}_p \phi \leq -\Delta^N_p \phi \leq g(x,u,\nabla \phi), \quad \text{provided} \, \, \nabla \phi (x) \neq 0.
	\end{align*}
	On the other hand, if $\nabla \phi(x) =0$, $-\mathcal{D}_p \phi \leq g(x,u,0)$ by definition. Hence, $u$ is a viscosity supersolution of $-\mathcal{D}_p u =g(x,u,\nabla u)$. 
\end{proof}
The following Lemma will be applied in the proof of the concavity, and it relies on the fact that the mapping $(q,A) \rightarrow \left \langle q, A^{-1}q \right \rangle$ is convex in $ S^+$ for each $q \in \mathbb{R}^n$. Here, $S^+$ consists of the symmetric positive definite matrices.

\begin{lemma}
Let $X_i \in S^+, \nu_i \in [0,1], i=1,...,k$, with $\sum_{i=1}^k \nu_i=1$. Then
\begin{align*}
\frac{1}{D_p \left( \sum_{i=1}^k \nu_i X_i  \right)^{-1}} \geq \sum_{i=1}^k \frac{\nu_i}{D_p X_i^{-1}}.
\end{align*}
\label{DiscreteOpprop}
\end{lemma}

\begin{proof}
In the appendix of [ALL] it was shown that $(q,A) \rightarrow \left \langle q, A^{-1}q \right \rangle$ is convex, 
\begin{align*}
\left \langle q, (\mu A_1 + (1-\mu)A_2)^{-1} q \right \rangle \leq \mu \left \langle q, A^{-1}q \right \rangle + (1-\mu)\left \langle q,A_2^{-1}q \right \rangle,
\end{align*}
for $q\in \mathbb{R}^n, A_1,A_2 \in S^+$ and $\mu \in [0,1]$. Consequently, 
\begin{equation}
D_p \left( \mu A_1 + (1-\mu)A_2 \right)^{-1} \leq \mu D_p A_1^{-1} + (1-\mu)D_p A_2^{-1}.
\label{alvaresRes}
\end{equation}
We label $c_1= D_p (X_1^{-1}), c_2 = D_p(X_2^{-1})$ and choose
\begin{align*}
A_1= \frac{X_1}{c_2}, \, A_2= \frac{X_2}{c_1}, \, \mu= \frac{\nu c_2}{\nu c_2 + (1-\nu)c_1}.
\end{align*}
With these choices,
\begin{align*}
D_p\left(  \nu X_1 + (1-\nu)X_2 \right)^{-1} = \frac{D_p \left( \mu A_1 + (1-\mu) A_2  \right)^{-1} }{\nu c_2 + (1-\nu) c_1}.
\end{align*}

Using inequality \eqref{alvaresRes} we find
\begin{align*}
\frac{1}{D_p \left( \nu X_1 + (1-\nu)X_2 \right)^{-1}} &= \frac{\nu c_2 + (1-\nu) c_1}{D_p \left( \mu A_1 + (1-\mu )A_2 \right)^{-1}} \\
&\geq \frac{\nu c_2 + (1-\nu)c_1}{\mu D_p (A_1^{-1}) + (1-\mu)D_p (A_2^{-1})} \\
&= \frac{\nu c_2 + (1-\nu)c_1}{\mu c_1 c_2 + (1-\mu) c_1 c_2} \\
& = \frac{\nu}{c_1} + \frac{1-\nu}{c_2}\\
&= \frac{\nu}{D_p (X_1^{-1})} + \frac{1-\nu}{D_p (X_2^{-1})}.
\end{align*}
By induction, the inequality in Lemma \ref{DiscreteOpprop} holds. 

\end{proof}

\textbf{Convex envelope}

The \textit{convex envelope} of a function $u: \Omega \rightarrow \mathbb{R}^n$ is defined as 
\begin{align*}
u_{**}(x) = \inf \left \{ \sum_{i=1}^k \mu_i u(x_i) \, : \, x_i \in \Omega, \,  \sum \mu_i x_i = x, \, \sum \mu_i=1, \, k \leq n+1, \mu_i \geq 0  \right\}.
\end{align*}

We are interested in the convex envelope of the square root, $v=-\sqrt{u},$ and we have the following result on what happens near the boundary of $\Omega$. 

\begin{lemma}
Let $u$ be a viscosity solution to $-\mathcal{D}_p u =1$ in a convex domain $\Omega$ that satisfies the interior sphere condition. Further let $x \in \Omega$, $x_1,...,x_k \in \bar \Omega$, $\sum_{i=1}^k \mu_i=1$ with
\begin{align*}
x= \sum_{i=1}^k \mu_i x_i, \quad u_{**}(x) = \sum_{i=1}^k \mu_i u(x_i).
\end{align*}
Then $x_1,...,x_k \in \Omega$. 
\label{boundary}
\end{lemma}

\begin{proof}
Since $u$ is, in particular a viscosity supersolution to $-\Delta_p^N u =1$, Lemma 3.2 in [K] gives the result. 
\end{proof}

\section{Concave square-root solutions. }
First, we examine which equation $v= -\sqrt{u}$ solves in the viscosity sense.
\begin{lemma}
A function $u\in USC(\bar{\Omega})$ is a positive viscosity subsolution of $-\mathcal{D}_p u=1$ with $u=0$ on $\partial \Omega$ if and only if $v=-\sqrt{u} \in LSC(\bar{\Omega})$, with $v=0$ on $\partial \Omega$, is a negative viscosity supersolution of
\begin{align*}
-\mathcal{D}_p v = \frac{1}{v}\left( (p-1)|\nabla v|^2 + \frac{1}{2}\right). 
\end{align*}
\label{squareqndom}
\end{lemma}

\begin{proof}
Let $u$ be a viscosity subsolution of $-\mathcal{D}_p u=1$. Take $\phi \in C^2(\Omega)$ such that for some $r>0$, 
\begin{align*}
0= (v- \phi) (x_0) < (v-\phi) (x), \quad \text{for all } \, x \in B_r (x_0),
\end{align*}
so that $v-\phi$ has a strict local minimum point at $x_0 \in \Omega$. Let $\psi(x)= \phi(x)^2 $. Then, since $v(x), \phi (x) <0 $, 
\begin{align*}
& (u-\psi)(x_0) = \left( v(x_0) - \phi(x_0)\right) \left( v(x_0) + \phi(x_0) \right)=0\\
& (u-\psi)(x) =\left( v(x)- \phi(x)\right) \left( v(x) + \phi(x) \right) <0.
\end{align*} 
Hence, $u-\psi$ has a strict local \textit{maximum} at $x_0$. We see that $\psi_{x_i}= 2\phi \phi_{x_i}$, $\psi_{x_i x_j} = 2\phi_{x_i}\phi_{x_j} + 2\phi \phi_{x_i x_j}$. Since $u$ is a viscosity subsolution we have at $x_0$,
\begin{align*}
1 &\geq -\mathcal{D}_p \psi \\
&=-2 \text{tr}\left( \nabla  \phi \otimes \nabla \phi + \phi D^2 \phi \right) - 2(p-2)\lambda_{\max} \left( \nabla  \phi \otimes \nabla \phi + \phi D^2 \phi    \right) \\
& \geq -2 |\nabla \phi|^2  -2\phi \Delta \phi - 2(p-2)\lambda_{\max} \left( \nabla \phi \otimes \nabla \phi  \right) -2\phi(p-2)\lambda_{\max} ( D^2 \phi ) \\
& = -2(p-1) |\nabla \phi|^2 - 2\phi \mathcal{D}_p \phi.
\end{align*} 
Dividing by $\frac{1}{2\phi(x_0)}$ gives $-\mathcal{D}_p \phi(x_0) \geq \frac{1}{2\phi(x_0)} \left(  (p-1) |\nabla \phi(x_0)|^2 + \frac{1}{2} \right)$, which shows that $v$ is a viscosity supersolution of 
\begin{align*}
-\mathcal{D}_p v = \frac{1}{v}\left( (p-1) |\nabla v|^2 + \frac{1}{2}   \right). 
\end{align*}
On the other hand, suppose $v\in LSC(\Omega)$ is a negative viscosity supersolution of $-\mathcal{D}_p v = \frac{1}{v} \left( (p-1)|\nabla v|^2 + \frac{1}{2}  \right)$. By Lemma \ref{connectionvisc}, $v$ is a viscosity supersolution of 
\begin{align*}
-\Delta^N_p v \geq \frac{1}{v} \left( (p-1)|\nabla v|^2 + \frac{1}{2} \right).
\end{align*}
Applying Lemma \ref{squaresolution} we see that $u = v^2 $ is a positive viscosity subsolution of 
\begin{align*}
-\Delta^N_p u =1. 
\end{align*}
A second application of Lemma \ref{connectionvisc} shows that $u$ is a viscosity subsolution of 
\begin{align*}
-\mathcal{D}_p u= 1. 
\end{align*}
\end{proof}

We now focus our attention on the convex envelope, $v_{**}$. It turns out that $v_{**}$ is a viscosity supersolution to the same equation as $v$. 
\begin{lemma}
Let $u\in USC(\bar \Omega)$ be a positive viscosity subsolution to $-\mathcal{D}_p u =1$ with $u=0$ on $\partial \Omega$ in a convex domain $\Omega$ that satisfies the interior sphere condition. If $v=-\sqrt{u}$, then $v_{**}$ is a negative viscosity supersolution to 
\begin{align*}
-\mathcal{D}_p v_{**} = \frac{1}{v_{**}} \left( (p-1)|\nabla v_{**}|^2 + \frac{1}{2}  \right) 
\end{align*}
with $v_{**}=0$ on $\partial \Omega$. 
\label{convenveqn}
\end{lemma}

\begin{proof}
According to [ALL, Lemma 4] we have $v_{**}=v=0$ on $\partial \Omega$ so we only have to show that $v_{**}$ is a viscosity supersolution. To this end, let $(q,A) \in  J^{2,-}v_{**}(x)$. By Lemma \ref{boundary} we can decompose $x$ in a convex combination of interior points, 
\begin{align*}
x= \sum_{i=1}^k \mu_i x_i, \quad v_{**}(x)= \sum_{i=1}^k \mu_i v(x_i), \quad \sum_{i=1}^k \mu_i =1,
\end{align*}
with $x_1,...,x_k \in \Omega$. 
By Proposition 1 in [ALL] there are $A_1,...,A_k \in S^+$ such that $(q,A_i) \in \bar J^{2,-} v(x_i)$ and
\begin{align*}
A-\epsilon A^2 \leq \left( \mu_1 A_1^{-1} + ... + \mu_k A_k^{-1}\right)^{-1} 
\end{align*}
for all $\epsilon>0$ small enough. Since $v$ is a viscosity supersolution, 
\begin{align*}
-D_p(A_i) \geq \frac{1}{v(x_i)} \left( (p-1)|q|^2 + \frac{1}{2} \right)
\end{align*}
Multiplying both sides with $\mu_i v(x_i)$ and a summation $i=1,...,k$ yields
\begin{align*}
-v_{**} (x) \leq \left( (p-1)|q|^2 + \frac{1}{2} \right) \sum_{i=1}^k \frac{\mu_i}{D_p(A_i)}.
\end{align*}
Using this inequality we find
\begin{align*}
&-D_p(A-\epsilon A^2) - \frac{1}{v_{**}(x)}\left((p-1)|q|^2 + \frac{1}{2} \right) \\
&\geq -D_p(A-\epsilon A^2) + \left( \sum_{i=1}^k \frac{\mu_i}{D_p(A_i)} \right)^{-1}.
\end{align*}
Lemma \ref{DiscreteOpprop} then gives
\begin{align*}
&-D_p(A-\epsilon A^2) - \frac{1}{v_{**}(x)}\left((p-1)|q|^2 + \frac{1}{2} \right) \\
&\geq -D_p(A-\epsilon A^2) + D_p \left( \sum_{i=1}^k \mu_i X_i^{-1}\right)^{-1} \geq 0
\end{align*}
since $A-\epsilon A^2 \leq \left( \sum_{i=1}^k \mu_i X_i^{-1}  \right)^{-1}$. Letting $\epsilon \rightarrow 0$ we see that 
\begin{align*}
-D_p (A) \geq \frac{1}{v_{**}(x)}\left( (p-1)|q|^2 + \frac{1}{2} \right)
\end{align*}
which shows that $v_{**}$ is a viscosity supersolution to 
\begin{align*}
-\mathcal{D}_p v_{**} = \frac{1}{v_{**}} \left((p-1)|\nabla v_{**}|^2 + \frac{1}{2} \right).
\end{align*}
\end{proof}

\textbf{Proof of Theorem \ref{maintheorem}. }
\begin{proof}
We have to show that $v=-\sqrt{u}$ is convex, making $\sqrt{u}$ concave,  if $u$ is a viscosity solution of
\begin{align}
\label{maineqndom}
\begin{split}
-\mathcal{D}_p u= 1 \quad &\text{in} \, \, \Omega ,
\\\quad u= 0 \quad &\text{on} \, \,  \partial \Omega.
\end{split}
\end{align}

Since $u$ is, in particular, a supersolution, it is positive. By Lemma \ref{convenveqn}, $v_{**}$ is a negative supersolution of 
\begin{align*}
-\mathcal{D}_p v_{**} = \frac{1}{v_{**}} \left( (p-1)|\nabla v_{**}|^2 +  \frac{1}{2} \right). 
\end{align*}

By Lemma \ref{squareqndom}  
\begin{align*}
 -\mathcal{D}_p \left( v_{**} \right)^2 \leq 1.
\end{align*}

We have found a subsolution of equation \eqref{maineqndom}. The comparison principle allows us to conclude that 
\begin{align*}
v_{**}^2 \leq u= v^2, \quad \text{in} \, \, \Omega.
\end{align*} 
But $v_{**} \leq v <0 $. Thus we must have $v_{**}=v$, showing that $v$ is convex. 
\end{proof}
\label{Nonlinear torsion problem}
\section{Log-concavity for the eigenvalue problem}
We proceed in the same manner as in section \ref{Nonlinear torsion problem}. The proofs of the following two Lemmas are similar to the proofs of Lemma \ref{squareqndom} and \ref{convenveqn}. We note that the interior sphere condition is not needed here, since $v=-\ln u$ converges to infinity on the boundary. This makes a similar version of Lemma \ref{boundary} redundant.

\begin{lemma}
	Assume that $\Omega$ is a convex domain in $\mathbb{R}^n$ and let $\lambda>0$. 
	A function $u\in USC(\bar{\Omega})$ is a positive viscosity subsolution to ${-\mathcal{D}_p u=\lambda u}$ with $u= 0 $ on $\partial \Omega$ if and only if $v=-\ln u \in LSC(\bar{\Omega})$ is a negative viscosity supersolution to
	\begin{align*}
	-\mathcal{D}_p v = -(p-1)|\nabla v|^2 - \lambda.
	\end{align*}
	\label{logequation}
\end{lemma}

\begin{lemma}
Assume that $\Omega$ is a convex domain in $\mathbb{R}^n$ and let $\lambda>0$. Let $u\in USC(\bar{\Omega})$ be a positive viscosity subsolution to $-\mathcal{D}_p u =\lambda u$ with $u= 0$ on $\partial \Omega$. If $v=-\ln u$, then $v_{**}$ is a viscosity supersolution to 
\begin{align*}
-\mathcal{D}_p v_{**} = -(p-1)|\nabla v_{**}|^2 - \lambda.
\end{align*}
\label{logstareqn}
\end{lemma}

\textbf{Proof of Theorem \ref{maintheoremeigenvalue}. }
\begin{proof}
Let $u$ be a positive viscosity solution of $-\mathcal{D}_p u =\lambda u$. Denoting $v=-\ln u$, Lemma \ref{logstareqn} gives that $v_{**}$ is a viscosity supersolution to 
\begin{align*}
-\mathcal{D}_p v_{**} = -(p-1)|\nabla v_{**}|^2 - \lambda.
\end{align*}
Lemma \ref{logequation} gives
\begin{align*}
-\mathcal{D}_p e^{v_{**}} \leq \lambda e^{-v_{**}}
\end{align*}
in the viscosity sense. By the comparison principle, 
\begin{align*}
e^{-v_{**}} \leq u = e^{-v}, \quad \text{in } \, \Omega.
\end{align*}
This together with the fact that $v_{**} \leq v$ shows that $v_{**}=v$, making $v$ a convex function and $\log u$ a concave function. 
\end{proof}

\section{Conclusion and further problems}
In this paper, we showed certain concavity properties of power functions for solutions of the homogeneous Dirichlet problem for the Dominative $p$-Laplace equation. This was due to the structure of the equation and its relation to the normalized $p$-Laplace operator. An interesting question is whether the parabolic version, $u_t= \mathcal{D}_p u$ has similar concavity properties and in what way it depends on the initial data. Further, for $n=2$, the equation can be explicitly written out, and it would be interesting to see a simple proof of the same result.

\end{document}